%% file: splitting-sa.tex
\documentclass{amsart}
\usepackage[margin=1in]{geometry}
\usepackage{verbatim}
\usepackage[textsize=scriptsize]{todonotes}

\input{preamble-sa.tex}

\myexternaldocument[Colimits:]{colimits-sa}
\myexternaldocument[Ops:]{operations-sa}

\title{Splitting results for normed motivic spectra}
\begin{document}
\maketitle

\begin{abstract}
We prove that the universal normed motivic spectrum of characteristic $2$ over a scheme on which $2$ is a unit, splits into a sum of motivic Eilenberg--MacLane spectra.
\end{abstract}
\setcounter{tocdepth}{1}
\tableofcontents

\input{splitting.tex}

\bibliographystyle{amsalpha}
\bibliography{powerops}

\end{document}

%% file: preamble-sa.tex
\input{preamble.tex}

\newenvironment{subappendices}{\appendix}{}

\usepackage{xstring}
\usepackage{xr}
\newcommand{\myexternaldocument}[2][] {{
\let\nl\newlabel
\def\nlxx##1##2##3##4##5##6{\nl{##1}{{#1##2}{}{}{}{}}}
\renewcommand\newlabel[2]{\IfBeginWith{##1}{tocindent}{}{\nlxx{##1}##2}}
\externaldocument{#2}
}}

%% file: preamble.tex
\usepackage{tikz}
\usetikzlibrary{matrix,arrows}
\usepackage{tikz-cd}
\usepackage{amsmath}
\usepackage{relsize}
\usepackage{amssymb}
\usepackage{amsthm}
\usepackage{amscd}
\usepackage{enumerate}
\usepackage{bbm}
\usepackage{etoolbox}
\usepackage{mathtools}
\usepackage{chngcntr}

\usepackage[utf8]{inputenc}

\newtoggle{final}
\toggletrue{final}

\numberwithin{equation}{section}

\newtheorem{proposition}[equation]{Proposition}
\newtheorem{corollary}[equation]{Corollary}
\newtheorem{lemma}[equation]{Lemma}
\newtheorem{theorem}[equation]{Theorem}
\newtheorem*{theorem*}{Theorem}
\newtheorem*{corollary*}{Corollary}
\newtheorem*{proposition*}{Proposition}
\newtheorem*{lemma*}{Lemma}
\theoremstyle{definition}
\newtheorem{definition}[equation]{Definition}

\newtheorem*{definition*}{Definition}
\newtheorem*{construction*}{Construction}

\theoremstyle{remark}
\newtheorem{remark}[equation]{Remark}

\newcommand{\id}{\operatorname{id}}
\newcommand{\Z}{\mathbb{Z}}

\newcommand{\NN}{\mathrm{N}}
\newcommand{\Q}{\mathbb{Q}}
\newcommand{\F}{\mathbb{F}}

\let\scr=\mathcal
\let\bb=\mathbb
\newcommand{\Gm}{{\mathbb{G}_m}}

\def\A{\bb A}

\def\P{\bb P}

\def\R{\bb R}
\newcommand{\1}{\mathbbm{1}}

\newcommand{\veff}{{\text{veff}}}

\newcommand{\SH}{\mathcal{SH}}

\def\ph{\mathord-}

\newcommand{\lra}[1]{\langle #1 \rangle}

\DeclareMathOperator*{\colim}{colim}

\let\lim=\relax
\DeclareMathOperator*{\lim}{lim}
\def\Map{\mathrm{Map}}

\def\Spec{\mathrm{Spec}}

\def\NAlg{\mathrm{NAlg}}

\def\Tor{\mathrm{Tor}}

\def\Sect{\mathrm{Sect}}
\def\PSh{\mathcal{P}}

\def\Span{\mathrm{Span}}

\def\Spc{\mathcal{S}\mathrm{pc}{}}

\def\cof{\mathrm{cof}}
\def\cofib{\mathrm{cofib}}

\newcommand{\wequi}{\simeq}
\newcommand{\Mod}{\mathcal{M}\mathrm{od}}
\def\adj{\rightleftarrows}

\DeclareRobustCommand{\ul}{\underline}

\newcommand{\tr}{\mathrm{tr}}

\def\op{\mathrm{op}}

\let\cat=\mathrm
\def\Sm{{\cat{S}\mathrm{m}}}
\def\Aff{{\cat{A}\mathrm{ff}}}
\def\SmAff{{\cat{S}\mathrm{mAff}}}
\def\Sch{\cat{S}\mathrm{ch}{}}

\def\FEt{\mathrm{FEt}{}}
\def\fet{\mathrm{f\acute et}}
\def\all{\mathrm{all}}

\def\GL{\mathrm{GL}}

\def\Sq{\mathrm{Sq}}
\def\mot{\mathrm{mot}}

\def\H{\mathrm{H}}

\def\NSym{\mathrm{NSym}}
\newcommand{\sslash}{\smash{\mathbin{/\mkern-6mu/}}}

\newcommand{\fpsr}[1]{[\![ #1 ]\!]}

\iftoggle{final} {
\newcommand{\NB}[1]{}
\newcommand{\tombubble}[1]{}
\newcommand{\tom}[1]{}
\newcommand{\elden}[1]{}
\newcommand{\eldenbubble}[1]{}
\newcommand{\jeremiah}[1]{}
\newcommand{\sjeremiah}[1]{}
\renewcommand{\todo}[1]{}
}{ % else
\newcommand{\NB}[1]{\todo[color=gray!40]{#1}}
\newcommand{\tombubble}[1]{\todo[color=green!40]{#1}}
\newcommand{\tom}[1]{{\color{green!60!black}#1}}
\newcommand{\elden}[1]{{\color{blue!60!black}#1}}
\newcommand{\eldenbubble}[1]{\todo[color=blue!40]{#1}}
\newcommand{\jeremiah}[1]{{\color{red!60!black}#1}}
\newcommand{\sjeremiah}[1]{{\todo[color=red!40]{#1}}}
}

\author{Tom Bachmann}
\address{Department of Mathematics, University of Oslo, Oslo, Norway}
\email{tom.bachmann@zoho.com}

\author{Elden Elmanto}
\address{Department of Mathematics, Harvard University, Cambridge, MA, USA}
\email{eldenelmanto@gmail.com}

\author{Jeremiah Heller}
\address{Department of Mathematics, University of Illinois, Urbana-Champaign, IL, USA}
\email{jbheller@illinois.edu}

\date{\today}

%% file: splitting.tex
% vim: spelllang=en spell
\section{Introduction}
A classical result of Steinberger \cite[III.4.1]{hinfty}, implies that 
the universal $\mathcal{E}_{\infty}$-algebra of characteristic $2$ is a wedge of shifts of the topological Eilenberg-MacLane spectrum $\H\Z/2^{\rm top}$. We prove the following generalization, about $\1 {\sslash} 2$, the universal normed motivic spectrum of characteristic $2$ (see Definition~\ref{def:char}). 
\begin{theorem}[see Theorem \ref{thm:1-ss-2-splits}]
	Let $S$ be a scheme with $1/2\in \mathcal{O}(S)$. Then  
	$\1 \sslash 2$ 
	splits into a wedge of shifts and twists of $\H\Z/2$.
\end{theorem}

The proof of this result relies critically on computational input coming from power operations in normed motivic spectra, constructed  in \cite{colimits} and \cite{operations}.

\subsection*{Outline of the paper}
We begin in \S\ref{sec:MM} by establishing a generalized version of the Milnor-Moore theorem and apply it to obtain a proto-splitting result,  which forms the basis for our main splitting result. In \S\ref{sec:motive_of_free} we compute the motive of the free normed spectrum on various spheres. Finally in \S\ref{sec:splitting_of_1ss2}, using power operations on $\H\Z/2_{**}(\1\sslash 2)$, we deduce our main splitting result. 

There are two appendices. In \S\ref{sec:A1topology}, a couple of useful facts concerning elements of $[\1,\1]_{\SH(S)}$ are recorded and \S\ref{sec:basechange} concerns base change functors and normed spectra, in particular that the formation of free normed spectra commutes with base change.

\subsection*{Acknowledgements}
We are grateful to Jeremy Hahn for introducing us to the elegant approach to splitting problems using the Milnor--Moore theorem.

Heller was partially supported by the NSF under award DMS-1710966.
\subsection*{Notation and conventions}
This paper is the third in a series, preceded by \cite{colimits} and \cite{operations}.
References to the previous ones are written  in the form \ref{ex:1} or \ref{thm:steinberger}.
We freely use notation and conventions from the previous papers. In addition, we use the following notation.

We write $\NAlg(\SH(S))$ for the category denoted by $\NAlg_{\Sm_S}(\SH^\otimes)$ in \cite[\S7]{norms}.
More generally, given a normed category $\scr D$ over $S$ in the sense of \cite[\S6.2]{norms}, we write $\NAlg(\scr D(S))$ for the analogous category of normed objects.
In particular if $H \in \NAlg(\SH(S))$ then $\Mod_H$ upgrades to a normed category, and we obtain $\NAlg(\Mod_H) \wequi \NAlg(\SH(S))_{H/}$ \cite[Proposition 7.6(4)]{norms}.
The forgetful functor $\NAlg(\SH(S)) \to \SH(S)$ has a left adjoint, which we denote by $\NSym$.
Similarly the forgetful functor $\NAlg(\Mod_H) \to \Mod_H$ has a left adjoint, denoted $\NSym_H$.

For $a \in \scr O^\times(S)$, we write $\lra{a} \in [\1, \1]_{\SH(S)}$ for the automorphism induced by multiplication by $a$ on $\Gm$.
We also put 
\[
 n_\epsilon = 1 + \lra{-1} + 1 + \dots + \lra{(-1)^{n-1} 1}, 
 \] where the sum consists of $n$ terms.

\section{A generalized Milnor-Moore theorem}\label{sec:MM}

All of the (graded) Hopf algebroids and comodule algebras considered in this paper will be (graded) commutative.
Recall the following basic structural result \cite[Corollary A1.1.18]{ravenel1986complex}.

\begin{theorem}[Milnor-Moore] \label{thm:milnor-moore}
Let $k$ be a field, $A$ a graded, connected Hopf algebra over $k$ and $M$ a graded, connected $A$-comodule-algebra. Suppose that there exists a surjection $M \to A$ of graded comodule algebras. Then $M$ is cofree.

More specifically, let $PM := \{m \in M \mid \psi(m) = m \otimes 1 \}$. If $\alpha: M \to PM$ is any graded $k$-linear splitting of the canonical inclusion $PM \hookrightarrow M$, then $\alpha$ cofreely generates $M$.
\end{theorem}
In other words there exists a morphism of $k$-vector spaces $M \to C$ (where $C \wequi PM$) such that the composite $M \to M \otimes_k A \to C \otimes_k A$ is an isomorphism (of vector spaces and hence comodule algebras).
\begin{proof}
The ``more specifically'' part is the content of the first paragraph of the proof of \cite[A1.1.17]{ravenel1986complex}.
\end{proof}

We wish to generalize this result to the situation where $k$ need not be a field, but rather is a (bigraded) ring. It is then natural to consider Hopf algebroids instead of Hopf algebras. There is a generalization of the above theorem for Hopf algebroids \cite[Theorem A1.1.17]{ravenel1986complex}, but this is not quite what we want.

To explain our result, we first need a small digression on scalar extension of Hopf algebroids. Suppose that $(R, \Gamma)$ is a Hopf algebroid and $R \to R'$ is a homomorphism of commutative rings. We would like to say that $(R', R' \otimes_R \Gamma)$ is a Hopf algebroid, but this does not make sense. Indeed $\Gamma$ is an $R$-bimodule, so on which side are we even tensoring? 

Considering the stack $\scr M$ represented by $(R, \Gamma)$, we note that we are attempting to form a new stack affine over $\scr M$. This corresponds to a quasi-coherent sheaf of commutative algebras on $\scr M$, or equivalently a comodule algebra under $R$. Thus we see \eldenbubble{I don't see how the stack point of view explains the conditions above; don't we just use the left $R$-module on $\Gamma$ so the only thing to construct is the right unit? \tom{It explains that we need to make $R'$ into a comodule algebra under $(R, \Gamma)$, i.e. why we need a new right unit. Of course it is also algebraically obvious that you need this extra structure. Perhaps the more importantly, the stacky point of view nicely explains that in the case of a closed immersion, there is at most one compatible right unit, and when it exists.}} that in order to extend the scalars, we need to to also provide a morphism $\eta_r': R' \to R' \otimes_R \Gamma$ (the new right unit), satisfying certain compatibilities. If $R \to R'$ is surjective with kernel $I$, then there is at most one such morphism $\eta_r'$, which exists if and only if $I\Gamma = \Gamma I$ (in other words, a closed subset of $\Spec(R)$ descends to a closed subset of $\scr M$ if and only if the two corresponding closed subsets in $\Spec(\Gamma)$ are equal).

\begin{corollary} \label{corr:milnor-moore}
Let $k$ be a field, $R$ a graded-commutative $k$-algebra, $R_+ \subset R$ a graded ideal such that $R/R_+  \wequi k$. Let $(R, \Gamma)$ be a Hopf algebroid and $M$ a graded comodule algebra under $(R, \Gamma)$. Suppose that
\begin{enumerate}
\item $R_+ \Gamma = \Gamma R_+$,
\item there is a surjection $\phi: M \to \Gamma$ of graded comodule algebras,
\item $M, \Gamma$ are projective as $R$-modules, and
\item $\Gamma/R_+\Gamma$ and $M/R_+M$ are connected.
\end{enumerate}
Then the following hold.
\begin{enumerate}[(a)]
\item There exists a free, graded, connected $R$-module $W$ and an $R$-module map $M \to W$ such that the induced map 
\[
 M/R_+ M \to M/R_+ M \otimes_{k} \Gamma/R_+ \Gamma \to W/R_+W \otimes_{k} \Gamma/R_+ \Gamma 
\] 
is an isomorphism.
\item If $R_+$ is the augmentation ideal and $\Gamma, M$ are bounded below, then for any $M \to W$ as in (a), the induced map $M \to W \otimes_{R} \Gamma$ is an isomorphism.
\end{enumerate}
In particular, under these assumptions, $M$ is cofree.
\end{corollary}
Note that the antipode $c: \Gamma \xrightarrow{\wequi} \Gamma$ is an isomorphism between the right and left $R$-module structures on $\Gamma$, so in (3) it does not matter if we view $\Gamma$ as a left or right $R$-module\NB{really?}.
\begin{proof}
For a left $R$-module $N$, we put $\bar{N} := N/R_+N = \bar{R} \otimes_R N$.

(a) It follows from (1),(4) that $(k, \bar{\Gamma})$ is a graded, connected Hopf algebroid, and hence Hopf algebra. Moreover $\bar{M}$ is a graded connected $\bar{\Gamma}$-comodule algebra, and $\bar{\phi}$ is a surjection of comodule algebras. Hence we may apply Theorem \ref{thm:milnor-moore} to find a surjection of graded $k$-vector spaces $\bar{M} \to V$ such that $\bar{M} \to \bar{M} \otimes_k \bar{\Gamma} \to V \otimes_k \bar{\Gamma}$ is an isomorphism. We claim that there exists a connected free graded $R$-module $W$ together with an isomorphism $\bar{W} \wequi V$ and a commutative diagram as follows
\begin{equation}\label{eq:milnor-moore-diagram}
\begin{CD}
M @>>> W \\
@VVV        @VVV \\
\bar{M} @>>> V.
\end{CD}
\end{equation}
Indeed, we can let $W:= R \otimes_k V$ so that we have a canonical surjection $W \rightarrow V \cong R/R_+ \otimes_R W$. Since $M$ is projective, the vertical arrow in~\eqref{eq:milnor-moore-diagram} exists.

(b) We show that $\psi: M \to M \otimes_R \Gamma \to W \otimes_R \Gamma$ is an isomorphism of graded $R$-modules. Since $\bar{\psi}$ is an isomorphism by assumption, the graded Nakayama lemma \cite[Exercise 4.6]{eisenbud2013commutative} implies that $\psi$ is surjective. Form the exact sequence 
\[
0 \to K \to M \to W \otimes_R \Gamma \to 0.
\] 
We get an induced exact sequence 
\[
\Tor_1^R(W \otimes_R \Gamma, \bar{R}) \to \bar{K} \to \bar{M} \to \overline{W \otimes_R \Gamma}.
\] Since $W$ and $\Gamma$ are projective so is $W \otimes_R \Gamma$, so the Tor vanishes and $\bar{K} = 0$. It follows that $K=0$, by the graded Nakayama lemma again. This concludes the proof.
\end{proof}

\subsection{Application to splitting problems}
In order to apply the previous result effectively, we will need to produce \emph{non-standard} gradings.
We produce these by the following maneuver. Suppose given a homomorphism $\alpha: \Z \times \Z \to \Z$. Then given a bigraded group $A_{**}$, we put $A_{\star = i} = \bigoplus_{\alpha(m,n) = i} A_{m,n}.$ We have $\alpha((1,0)) = a$ and $\alpha((0,1)) = -b$ for some $a,b \in \Z$, which we suppose are not both zero. The line $L = \{am-bn = 0\} \subset \Z \times \Z$ is sent to zero by $\alpha$. Suppose further that $a > 0$. Then $L$ has a canonical orientation and cuts $\Z \times \Z$ into two half planes; $\alpha$ sends the right half plane to $\Z_{> 0}$ and the left half plane to $\Z_{<0}$. In particular $A_{**}$ is connective in the $\star$-grading if and only if all non-zero groups are contained in the right half plane of some translate of $L$, and $A_{**}$ is connected if and only if all non-zero groups are contained in the right half plane of $L$, with the exception of $A_{0,0} = k$. Note also that up to rescaling (which does not change any of the connectivities), the information in $\alpha$ is the same as the information in $L$, and any line through the origin with rational slope can occur. More precisely, if we define an equivalence relation on the set of all non-zero homomorphisms $\alpha:  \Z \times \Z \to \Z$ by declaring that $\alpha \simeq \alpha'$ if and only if 
\[
\frac{\alpha((1,0))}{\alpha((0,1))} = \frac{\alpha'((1,0))}{\alpha'((0,1))};
\]
then there is a bijection between all non-zero homomorphisms $\{\alpha:  \Z \times \Z \to \Z \}$ modulo this equivalence relation with lines $\{ L \subset  \Z \times \Z\}$ through the origin 

\begin{definition}
Let $L$ be an oriented line of rational slope through the origin $\R^2$. We call any of the gradings constructed above a $\star$-grading corresponding to $L$, and we say that $A_{**}$ is connected (respectively connective) with respect to $L$ if $A_\star$ is connected (respectively connective). We also say that $A_{**}$ has a vanishing line $L$ if $A_\star = 0$ for $\star \le 0$.
\end{definition}

As noted above, all $\star$-gradings corresponding to $L$ differ only by rescaling, and in particular being connected (or connective) is independent of the choice of $\star$-grading corresponding to $L$.

\begin{theorem} \label{thm:detect-splitting}
Let $S$ be connected and essentially smooth over a Dedekind domain, and assume that all of its residue fields are perfect and of finite virtual 2-étale cohomological dimension. Let $E \in \SH(S)$ be a homotopy commutative ring spectrum with $2=0$. Suppose that the following hold
\begin{enumerate}
\item $E$ is connective and slice-connective (that is, $E \in \SH(S)^\veff(n)$ for some $n \in \Z$)
\item $E \wedge \H\Z/2$ is a summand of a split Tate motive, and $\ul{\pi}_{**}(E \wedge \H\Z/2)$ is connected with respect to a line of slope $0 < q < 1$ through the origin
\item there is a ring map $E \to \H\Z/2$ such that the induced map $\ul{\pi}_{**}(\H\Z/2 \wedge E) \to \ul{\pi}_{**}(\H\Z/2 \wedge \H\Z/2)$ is surjective.
\end{enumerate}
Then $E$ is equivalent to a wedge of shifts and twists of $\H\Z/2$.
\end{theorem}
\begin{proof}
Put $R = \H\Z/2_{**}$ and $R_+ = \H\Z/2_{*,*<0}$. Then $R_+$ is a bigraded ideal and $R/R_+ \wequi \Z/2$, by Proposition \ref{prop:HZ-dedekind}(4). Let $\scr A = \H\Z/2_{**} \H\Z/2$; then $(\H\Z/2_{**}, \scr A)$ is a bi-graded Hopf algebroid. The map $\H\Z/2_{**} E \to \scr A$ induced by (3) is a surjective morphism of comodule algebras. We show in Lemma \ref{lemm:steenrod-R+} below that $R_+ \scr A^\vee = \scr A^\vee R_+$; here $\scr A^\vee$ denotes the Steenrod algebra. Taking duals\todo{is this really legit?} as explained in \S\ref{subsub:composition-algebra}, we find that $R_+ \scr A = \scr A R_+$.

Increasing $q$ slightly if necessary, we may assume that $q \in \Q$. Now the line from (2) has rational slope; pick a corresponding $\star$-grading. From now on we treat all bigraded objects as $\star$-graded.
By assumption (2), $\H\Z/2_{\star} E$ is projective. Next we also claim that for any $S' \rightarrow S$, the canonical map 
$(\H\Z/2_\star E_{S'})/R_+ \rightarrow (\H\Z/2_\star E)/R'_+$ is an isomorphism, i.e., $(\H\Z/2_{\star} E) /R_+$ is independent of $S$.  Since $R/R_+ = \Z/2$, it is independent of $S$, and thus the claim holds when $E \wedge \H\Z/2$ is a split Tate motive, or a summand thereof. Thus the claim follows from assumption (2). Therefore, we can conclude that $(\H\Z/2_{\star}E)/R_+$ is connected, since this holds after pullback to any henselian local scheme, by (2).
 
Since $\phi$ is surjective, this implies that $\scr A/R_+$ is $\star$-connected as well. By Theorem \ref{thm:HZ-dual-steenrod}(1) we know that $\scr A_{\star}$ is projective.
Applying Corollary \ref{corr:milnor-moore}(a) to $M=\H\Z/2_\star E$ and $\Gamma = \H\Z/2_\star \H\Z/2$, we obtain a split Tate motive $W$ and a map $\alpha': M \to \H\Z/2_{**} W$.
Writing $E \wedge \H\Z/2$ as a sum of Tate motives, we can lift $\alpha'$ to $\tilde\alpha: E \wedge \H\Z/2 \to W$.
Composing with the unit map $E \to E \wedge \H\Z/2$, we obtain $\alpha: E \to W$.
We claim this is an equivalence.
By Lemma \ref{lemm:conservative} below, it suffices to show that $\alpha \wedge \H\Z/2$ is an equivalence.
Since equivalences can be checked on $\ul\pi_{**}$, we may assume that $S$ is henselian local, and it suffices to show $\H\Z/2_{**}(\alpha): \H\Z/2_{**}E \to \H\Z/2_{**} W$ is an equivalence.
By construction, this coincides with the map \[ M = \H\Z/2_\star E \to M \otimes \Gamma = \H\Z/2_\star E \otimes_{\H\Z/2_\star} \H\Z/2_\star \to \H\Z/2_\star W. \]
We prove this is an isomorphism by appeal to Corollary \ref{corr:milnor-moore}(b).
We must show that $R_+$ is the augmentation ideal and $M, \Gamma$ are $\star$-connective.
Since $S$ is henselian local, the first statement follows from $0<q<1$ and Proposition \ref{prop:HZ-dedekind}(2), and the second statement follows from (2, 3).
\end{proof}

\begin{lemma}\label{lemm:steenrod-R+}
Let $\scr A^{**}$ denote the motivic Steenrod algebra over $S$ and $R_+$ the augmentation ideal of $\H\Z/2^{**}$. Then $R_+ \scr A^{**} = \scr A^{**} R_+$.
\end{lemma}
\begin{proof}
The algebra $\scr A^{**}$ is generated as a left $\H\Z/2^{**}$-module by the admissible monomials $\Sq^I$, by Theorem \ref{thm:HZ-dual-steenrod}(4). It thus suffices to prove that for every $i \ge 1$ and $x \in R_+$ we have $x \Sq^i \in \scr A^{**} R_+$ and $\Sq^i x \in R_+ \scr A^{**}$. By formula \eqref{eq:beta-x} of Example \ref{ex:composition-cartan}, we have $\beta x - x \beta = \beta(x) \in R_+$. This deals with $i = 1$. Formula \eqref{eq:Sq-x} of Example \ref{ex:composition-cartan} directly implies that $\Sq^{2n} x \in R_+ \scr A$. We can rearrange the formula to
\[ x \Sq^{2n} = \sum_{a + b = n, a > 0}\Sq^{2a}(x) \Sq^{2b} + \sum_{a+b=n-1} \tau \Sq^{2a+1}(x) \Sq^{2b+1} + \Sq^{2n} x. \]
From this $x \Sq^{2n} \in \scr A R_+$ follows by induction on $i=2n$.
\end{proof}

\begin{lemma}\label{lemm:conservative}
Let $S$ be a scheme of finite dimension, all of whose residue fields are of characteristic $\ne 2$, and of finite virtual 2-étale cohomological dimension. Let $\scr C(S) \subset \SH(S)$ denote the subcategory of connective, slice-connective spectra such that $2^n \id_E =0$ for some $n$. Then $\wedge \H\Z/2$ is a conservative functor on $\scr C(S)$.
\end{lemma}
\begin{proof}
If $f: S' \to S$ is any morphism, then $f^* \scr C(S) \subset \scr C(S')$. Since pulling back to the residue fields of $S$ is conservative by \cite[Proposition B.3]{norms}, we may assume that $S$ is the spectrum of a field of finite virtual 2-étale cohomological dimension and characteristic $\ne 2$.
Note that the exponential characteristic of $k$ is invertible on objects of $\scr C$ (since it is not $2$), and if $k$ is of positive characteristic, then it is of finite $2$-étale cohomological dimension (because virtual and actual $2$-étale cohomological dimension agree in positive characteristic\todo{ref}).
Thus by \cite[Corollary 2.1.7]{elmanto2018perfection} we may assume $k$ perfect, and (now) we may appeal to \cite[Theorem 16 and Lemma 19]{bachmann-hurewicz}, whence the functor $\wedge \H\Z$ is conservative on $\scr C(k)$.
Now suppose that $E \wedge \H\Z/2 = 0$ and $2^n \id_E = 0$.
Then $0 = E \wedge \H\Z/2^n$ (being obtained as a finite extension of copies of $E \wedge \H\Z/2$), and $E \wedge \H\Z/2^n \wequi \cofib(E \wedge \H\Z \xrightarrow{2^n} E \wedge \H\Z) \wequi E \wedge \H\Z \vee E \wedge \H\Z[1]$.
Thus $E \wedge \H\Z = 0$, in other words, $-\wedge \H\Z/2$ is also conservative on $\scr C(k)$.
\end{proof}

\section{Motives of free normed spectra}\label{sec:motive_of_free}
We now compute $D_2^{\mot}(S^{p,q}) \wedge \H\Z/2$ for many values of $(p,q)$.
This was essentially already done in the proof of Lemma \ref{lemm:squaring-class}.

\begin{lemma} \label{lemm:exceptional-spectrum-additive-structure}
	\,
	\begin{enumerate}
\item For $E \in \SH(S)$ there is a cofiber sequence
\[ \Sigma E \wedge E \to \Sigma D^{\mot}_2(E) \to D^{\mot}_2(\Sigma E). \]

\item Let $m \in \Z$ and $k \ge -1$ and assume that $1/2 \in S$. There is an equivalence
\begin{gather*}
  D_2^\mot(S^{2m-k, m}) \wedge \H\Z/2 \wequi\\
  T^{2m} \wedge S^{-k} \wedge \left(S^{-k} \vee S^{-k+1} \vee \dots S^{-1} \vee S^{0} \vee \Gm \vee \bigvee_{n \ge 1} (S^0 \vee \Gm) \wedge T^n\right) \wedge \H\Z/2.
\end{gather*}
(Here the sum $S^{-k} \vee S^{-k+1} \vee \dots S^{-1} \vee S^{0}$ is treated as empty if $k = -1$.)
\end{enumerate}
\end{lemma}
\begin{proof}
By Corollaries \ref{corr:D-mot-excisive} and \ref{corr:D-mot-reduced}, the functor $D^\mot_2$ is $2$-excisive and reduced. By Proposition \ref{prop:D2-mot-cr2}, the cross effect is given by $cr_2(\Sigma E, \Sigma E) \wequi E \wedge E$. The cofiber sequence is now obtained from Proposition \ref{prop:fibn-sequence}.

Now we prove the second half. We work in the category of $\H\Z/2$-modules and put $D_2 = D_2^\mot$. We reduce to $m=0$ by using Proposition \ref{prop:Dn-thom-iso}. The case $k=0$ is a special case of Proposition \ref{prop:bmu}.

For the case $k=-1$, we consider the cofiber sequence with $E = S^0$ to get 
\[
\Sigma S^0 \xrightarrow{f} \Sigma D_2(S^0) \to D_2(S^1).
\] By Proposition \ref{prop:Dn-thom-iso}, we have a decomposition
\[
D_2(S^0) \simeq S^0 \oplus \Gm \oplus T \oplus \dots,
\] and consequently $f$ corresponds to maps $f_0: \Sigma S^0 \to \Sigma S^0$, $f_1: \Sigma S^0 \to \Sigma \Gm$, and so on (only finitely many of which are non-zero). We claim that $f_0$ is an isomorphism; this implies that $D_2(S^1) = \cof(f)$ is equivalent to the sum of the remaining terms\NB{ref?}, which was to be shown. We may assume that $S = \Spec(\Z[1/2])$, everything being stable under base change. Then we have $[S^0, S^0]_{\H\Z/2} = \Z/2$ (by Proposition \ref{prop:HZ-dedekind}(4)), so $f_0$ is an equivalence as soon as it is non-zero. This we may check by complex realisation, using Example \ref{ex:Dmot-betti-realisation} and the fact that the homology of $D_2(S^i)$ is known classically for all $i$.\footnote{In fact classically $D_2(S^i)$ is the Thom complex of a virtual (real) vector bundle of rank $2i$ on $BC_2\wequi \R\P^\infty$, and hence $D_2(S^i) \wedge \H\Z/2 \wequi \Sigma^{2i} \R\P^\infty \wedge \H\Z/2$.}

Now we deal with $k > 0$. We do this by induction; suppose the result has been proved for some $k$. We consider the cofiber sequence with $E = S^{-k-1}$ to get 
\[
\Sigma S^{-2(k+1)} \to \Sigma D_2(S^{-k-1}) \to D_2(S^{-k}) \xrightarrow{g} \Sigma^2 S^{-2(k+1)} \wequi S^{-2k}.
\] The result will follow if we can show that $g=0$, whence 
\[
D_2(S^{-k})  \simeq \Sigma D_2(S^{-k-1})  \vee S^{-2k},
\] and the claim follows from the inductive hypothesis. For this we note that $D_2(S^{-k}) = S^{-2k} \oplus S^{-2k+1} \oplus \dots \oplus S^{-k} \oplus S^{-k} \wedge \Gm \oplus \dots$, by induction. Thus $g$ corresponds to maps $g_0: S^{-2k} \to S^{-2k}$, $g_1: S^{-2k+1} \to S^{-2k}$, \dots, $g_k: S^{-k} \to S^{-2k}$, $g_{k+1}: S^{-k} \wedge \Gm \to S^{-2k}$, and so on. We may again assume that $S=\Spec(\Z[1/2])$ and hence $g_1, \dots, g_k = 0$ by Proposition \ref{prop:HZ-dedekind}(4) and $g_{k+1}, g_{k+2}, \dots = 0$ by Proposition \ref{prop:HZ-dedekind}(1). Again $g_0 \in [S^{-2k}, S^{-2k}]_{\H\Z/2} = \Z/2$ and so we may check that $g_0 = 0$ after complex realisation, where it follows from classical results.
\end{proof}
\begin{remark} \label{rmk:identify-cofibration-map}
By construction, the map $\Sigma D_2(E) \to D_2(\Sigma E)$ coincides with the diagonal from Definition \ref{def:conorm} (for $F=S^1$).
\end{remark}

\begin{proposition} \label{prop:free-normed-spectra}
Let $k \ge -1, n \in \Z$ and $r = 2n-k$. Then $D_i^\mot(S^{r,n}) \wedge \H\Z/2$ is a summand of a split Tate motive. Moreover, there exists $1 > q > 1/2$\NB{FWIW, $q=5/9$ seems to work}{} such that if $S$ is essentially smooth over a Dedekind scheme and $i>0$, then $\ul{\pi}_{**}(\H\Z/2 \wedge D_i^\mot(S^1))$ has a vanishing line of slope $q$, through the origin.
\end{proposition}
\begin{proof}
Throughout we put $D_i := D_i^\mot$, and we work in the category of $\H\Z/2$-modules.
First we treat the case $k \ge 0$, which is much simpler. It is well-known that $\Sigma_i$ has Sylow $2$-subgroup a product of iterated wreath products of the group $\Sigma_2$ (we will review this below). It hence follows from Proposition \ref{prop:Dmot-iterated-wreath}, Proposition \ref{prop:Dmot-transfer} and Example \ref{ex:D-mot-prod} that $D_i(\Z/2[r](n))$ is a summand of $D_2^{(n_1)}(\Z/2[r](n)) \otimes D_2^{(n_2)}(\Z/2[r](n)) \otimes \dots \otimes D_2^{(n_s)}(\Z/2[r](n))$, for some numbers $n_1, \dots, n_s$. Here $D_2^{(i)}$ denotes the $i$-fold iterate of $D_2$. It hence suffices to prove that $D_2^{(i)}(\Z/2[r](n))$ is of the required form. More specifically we prove by induction on $i$ that $D_2^{(i)}(\Z/2[r](n))$ is a sum of terms of the form $\Z/2[2l-a](l)$. For $D_2^{(1)} = D_2$, this follows from Lemma \ref{lemm:exceptional-spectrum-additive-structure}. The induction step now follows from Proposition \ref{prop:D2-mot-cr2}, i.e. the fact that $D_2(A \oplus B) = D_2(A) \oplus D_2(B) \oplus A \otimes B$.

It thus remains to deal with $D_i(\Z/2[1])$. Denote by $\scr S \subset \{\Z/2[2l-a](l)\}_{l,a \in \Z}$ the subset of objects satisfying the following conditions:
\begin{enumerate}[(1)]
\item $a \ge -1$,
\item $l \ge 0$, $2l-a > 0$
\item $l/(2l-a) < q$, and
\item $l > 0$ unless $a = -1$.
\end{enumerate}
Here we pick $1 > q > 1/2$ large enough so that $T^2 \wedge \Gm \in \scr S$. Condition (3) of course means that $(2l-a,l)$ is below the line of slope $q$ through the origin. Note that, since $q > 1/2$
\begin{enumerate}[(a)]
\item whenever $X \in \scr S$ and $e \ge 0$, also $X[2e](e) \in \scr S$. Furthermore we claim that
\item if $X = \Z/2[2l-a](l) \in \scr S$, then $D_2(X) = \bigoplus_i X_i$, with each $X_i \in \scr S$.
\end{enumerate}
To see this, we examine cases. Suppose $a = -1$. By (a) and Proposition \ref{prop:Dn-thom-iso}, we may assume $l=0$. Now $D_2(X) = D_2(S^1) = T \oplus T[1] \oplus T^2 \oplus \dots$, by Lemma \ref{lemm:exceptional-spectrum-additive-structure}, which satisfies (1) to (4) (in particular (3) holds because $q>1/2$). Now consider $a=0$. By (a), we may assume that $l=1$. Then $D_2(X) = D_2(T) = T^2 \wedge (S^0 \oplus \Gm \oplus T \oplus \dots)$, which satisfies (1) to (4) because $T^2 \wedge \Gm \in \scr S$ by construction. Now let $a > 0$. Then $D_2(X) = X^{\otimes 2} \wedge (S^0 \oplus S^1 \oplus \dots \oplus S^{a-1} \oplus S^a \wedge D_2(S^0))$, again by Lemma \ref{lemm:exceptional-spectrum-additive-structure}. We note that $X^{\otimes 2} \in \scr S$, and hence also $X^{\otimes 2} \wedge S^e \in \scr S$ for all $e \ge 0$. Since $a>0$, $S^a \wedge D_2(S^0)$ is a sum of terms below the line of slope $1/2 < q$ through the origin, and hence $X^{\otimes 2} \wedge S^a \wedge D_2(S^0)$ is a sum of terms below the line of slope $q$. Hence (3) is satisfied. The other conditions are easy. Hence claim (b) is proved.

Put $\scr F = \NSym(\Z/2[1])$.
Here \[ \NSym \wequi \bigoplus_{i \ge 0} D_i \] denotes the free normed object functor \cite[Theorem 3.10]{bachmann-MGM}.
We shall prove by induction on $i \ge 1$ that there exists $E_i = \bigoplus_\alpha \bigotimes_\beta X_{\alpha, \beta}^i$ (where the sum is infinite but the tensor products are all finite) with $X_{\alpha, \beta}^i \in \scr S$, together with maps $\phi_{\alpha, \beta}^i: X_{\alpha, \beta}^i \to \scr F$ such that the induced map $\phi_i: E_i \to \scr F$ (using the multiplication in $\scr F$) is a split surjection onto the summand $D_i(\Z/2[1]) \subset \scr F$.

For $i = 1$, we just need $D_2(S^1) \in \scr S$, which holds by (b). Now suppose that the result has been proved for all $i' < i$. If $i$ is not a power of $2$, we may write $i = i_1 + i_2$ such that $i \choose i_1$ is odd (use Kummer's theorem\NB{ref}). It follows that $\Sigma_{i_1} \times \Sigma_{i_2} \subset \Sigma_i$ has odd index, and in particular the product $E_i := E_{i_1} \otimes E_{i_2} \to \scr F$ has the required properties (see again Proposition \ref{prop:Dmot-iterated-wreath}, Proposition \ref{prop:Dmot-transfer} and Example \ref{ex:D-mot-prod}). Now assume that $i$ is a power of $2$, say $i = 2^a$. Splitting a set of $2^a$ elements into two sets of $2^{a-1}$ elements which may also be swapped, we find $\Sigma_{2^{a-1}} \wr \Sigma_2 \subset \Sigma_{2^a}$. The index is $\frac{1}{2} {2^a \choose 2^{a-1}}$, which is again odd. It hence follows from Proposition \ref{prop:Dmot-iterated-wreath} and Proposition \ref{prop:Dmot-transfer} that $\phi': D_2(E_{2^{a-1}}) \xrightarrow{D_2(\phi_{2^{a-1}})} D_2(\scr F) \to \scr F$ is a split surjection onto $D_{2^a}(\Z/2[1])$. Unfortunately $D_2(E_{2^{a-1}})$ is not of the required form. It follows from Proposition \ref{prop:D2-mot-cr2} (and the fact that $D_2$ preserves filtered colimits) that
\[ D_2(E_{2^{a-1}}) = \bigoplus_{\alpha} D_2\left(\bigotimes_\beta X_{\alpha, \beta}^{2^{a-1}}\right) \oplus \bigoplus_{\alpha \ne \alpha'} \bigotimes_\beta X_{\alpha, \beta}^{2^{a-1}} \otimes \bigotimes_{\beta'} X_{\alpha', \beta'}^{2^{a-1}}. \]
By Lemma \ref{lemm:ur-cartan-formula}, for each $\alpha$ the map
\[ D_2\left( \bigotimes_\beta X_{\alpha, \beta}^{2^{a-1}} \right) \xrightarrow{D_2(\prod_\beta \phi_{\alpha, \beta})} D_2(\scr F) \to \scr F \]
factors as
\[ D_2\left( \bigotimes_\beta X_{\alpha, \beta}^{2^{a-1}} \right) \to \bigotimes_\beta D_2(X_{\alpha,\beta}^{2^{a-1}}) \xrightarrow{\bigotimes_\alpha D_2(\phi_{\alpha, \beta})} \bigotimes D_2(\scr F) \to \scr F. \]
It follows that if we put
\[ E_{2^a} = \bigoplus_{\alpha} \bigotimes_\beta D_2(X_{\alpha, \beta}^{2^{a-1}}) \oplus \bigoplus_{\alpha \ne \alpha'} \bigotimes_\beta X_{\alpha, \beta}^{2^{a-1}} \otimes \bigotimes_{\beta'} X_{\alpha', \beta'}^{2^{a-1}}, \]
Then $\phi': D_2(E_{2^{a-1}}) \to \scr F$ factors as $D_2(E_{2^{a-1}}) \to E_{2^a} \xrightarrow{\phi_{2^a}} \scr F$. In particular $\phi_{2^a}$ is a split surjection onto $D_{2^a}(\Z/2[1])$, since $\phi'$ is. Observation (b) now implies that $E_{2^a}$ is of the required form.

Finally we come to the vanishing line. Let us note that if $E, F$ have a vanishing line of slope $q$ through the origin, then so does $E \otimes F$, and similarly if $\{E_i\}_{i \in I}$ all have a vanishing line of slope $q$ through the origin, then so does $\bigoplus_i E_i$. Hence it suffices to show that if $X \in \scr S$, then $X$ has a vanishing line of slope $q$. We note that $S^0 \wedge \H\Z/2$ ``almost'' has a vanishing line of slope $q$: the sheaves $\ul{\pi}_{**} \H\Z/2$ are all located to the right \emph{or on} the line (namely $\ul{\pi}_{0,0}$ is on the line). This follows from Proposition \ref{prop:HZ-dedekind}(3). Now if $X = \Z/2[x](y) \in \scr S$ then $(x,y)$ is to the right of the vanishing line (by (1) and (3)), and hence so is $\ul{\pi}_{**}X$. This concludes the proof.
\end{proof}

\section{Normed spectra of characteristic $2$}\label{sec:splitting_of_1ss2}
The goal of this section is to prove Theorem~\ref{thm:1-ss-2-splits} which asserts that the universal normed spectrum of characteristic $2$ splits as a sum of shifts and twists of $\H\Z/2$. 
Given a map $x:\1\to \1$ in $\SH(S)$ we write $\bar{x}:\NSym(\1)\to \1$ for the map obtained by adjunction. 
The universal normed spectra of characteristic $x$ is defined as follows.

\begin{definition}\label{def:char}
Let $x:\1\to\1$ be a map in $\SH(S)$. The \emph{universal normed spectrum of characteristic $x$}, written $\1\sslash x$, is defined to be the pushout in $\NAlg(\SH(S))$
\begin{equation*}
\begin{CD}
\NSym(\1) @>{\bar{x}}>> \1 \\
@V{\bar{0}}VV     @VVV \\
\1    @>>>       \1 \sslash x.
\end{CD}
\end{equation*}
\end{definition}

We note that, essentially by construction, $\1 \sslash x$ is a normed spectrum of characteristic $x$, i.e., we have $x=0$ in $\pi_0(\1 \sslash x)$. The universality of $\1 \sslash x$ is expressed as follows.

\begin{lemma} \label{lemm:1-ss-2-versal}
Let $E \in \NAlg(\SH(S))$. The space of morphisms $\1 \sslash x \to E$ in $\NAlg(\SH(S))$ is equivalent to the space of paths in $\Map(\1, E)$ from $0$ to $x$. In particular, the unit map $\1 \to E$ factors (in general non-uniquely!) through $\1 \to \1 \sslash x$ if and only if $x = 0$ in $\pi_0(E)$.
\end{lemma}
\begin{proof}
By definition of $\1 \sslash x$  and since  $\1 \in \NAlg(\SH(S))$ is initial, we have a cartesian square
\[
\begin{tikzcd}
\Map_{\NAlg(\SH(S))}(\1 \sslash x, E) \ar{r} \ar{d} & * \ar{d}{x} \\
* \ar{r}{0} & \Map_{\SH(S)}(\1, E).
\end{tikzcd}
\]
This proves the first part. The second part is an immediate consequence.
\end{proof}

For the rest of this section, we will be interested only in $x = 2$ or $x=2_{\epsilon}$ (recall that $2_\epsilon = 1+\langle -1 \rangle$). We can approach the homology of $\1 \sslash 2$ via the following standard result, by setting $H = \H\Z/2$ and $x=2$.
\begin{lemma} \label{lemm:1-ss-2-identification}
Let $H \in \NAlg(\SH(S))$, $x \in \pi_0(\1)$ with $H_*(x) = 0$. Then $\1 \sslash x \wedge H \wequi \NSym(S^1) \wedge H$.
\end{lemma}
\begin{proof}
The functor $\NAlg(\SH(S)) \to \NAlg(\Mod_{H}), E \mapsto E \wedge H$ is left adjoint to the forgetful functor, and hence preserves colimits and free objects\NB{details?}. Consequently there is a pushout square in $\NAlg(\Mod_{H})$ as follows
\begin{equation} \label{eq:square-of-algebras}
\begin{CD}
\NSym_{H}(H) @>{\bar{x}}>> H \\
@V{\bar{0}}VV     @VVV \\
H    @>>>       \1 \sslash 2 \wedge H.
\end{CD}
\end{equation}
Since $x=0 \in [H, H]_{H}$, in the above diagram we have $\bar x = \bar 0$. Note that $H \wequi \NSym_{H}(0)$. We claim that $\bar{0} = F(0)$, where $0: H \to 0$ is the unique map.
Indeed in any adjunction $F \dashv U$, given $A \xrightarrow{\alpha} B \xrightarrow{\beta} UC$, we have $(\beta \circ \alpha)^\dagger = \beta^\dagger \circ F(\alpha)$, as is easily verified using the triangle identities.
The claim follows by applying this to the factorization $\1 \to 0 \to \1$ and noting that the adjoint of $0 \to \1$ is the equivalence $F(0) \wequi \1$ (both objects being initial).

On the other hand, the category $\Mod_{H}$ we have the pushout square
\begin{equation} \label{eq:square-of-modules}
\begin{CD}
H @>0>> 0  \\
@V0VV      @VVV \\
0      @>>> S^1 \wedge H.
\end{CD}
\end{equation}
The claim that $\bar{0} = F(0)$ implies that applying the colimit-preserving functor $\NSym_{H}$ to \eqref{eq:square-of-modules} yields \eqref{eq:square-of-algebras}, up to equivalence.
The result follows.
\end{proof}

We have constructed $\1 \sslash 2$ to be $2$-torsion. However more is true. 
%Recall the element $2_\epsilon$ from Appendix \ref{app:topology}.
\begin{lemma} \label{lemm:normed-2-h-torsion}
Let $1/2 \in S$ and $E \in \NAlg(\SH(S))$. Then $E$ is $2$-torsion if and only if $E$ is $2_\epsilon$-torsion.
\end{lemma}
\begin{proof}
Let $S' = S[T]/(T^2 + 1)$; then $S'/S$ is finite étale. We shall use the following properties of $\NN_{S'/S}$:
\begin{itemize}
\item $\NN_{S'/S}(1) = 1$
\item $\NN_{S'/S}(0) = 0$
\item $\NN_{S'/S}(x+y) = \NN_{S'/S}(x) + \NN_{S'/S}(y) + \tr_{S'/S}(x\bar{y})$, where $\bar{y}$ denotes the canonical automorphism of $S'/S$ applied to $y$.
\end{itemize}
Here the second and third properties follow from the distributivity law\NB{details?}.

Suppose that $E$ is $2$-torsion or $2_\epsilon$-torsion.
In $S'$, $-1$ is a square and so $2_\epsilon = 2$ by Lemma \ref{lemm:lra-squares}, so in particular $2=0$ in $E|_{S'}$. Thus we get \[ 0 = \NN_{S'/S}(0) = \NN_{S'/S}(2) = \NN_{S'/S}(1+1) = 2\NN_{S'/S}(1) + \tr_{S'/S}(1). \] We have $\tr_{S'/S}(1) = \lra{2}2_\epsilon$ by Lemma \ref{lemm:trace-of-Zi} and hence \[ 0 = 2 + \lra{2}2_\epsilon. \] Since $\lra{2}$ is a unit, the result follows.
\end{proof}

\begin{lemma} \label{lemm:steenrod-action-moore-spectrum}
The canonical map $\alpha: \1/2_\epsilon \to \H\Z/2$ induces $\alpha_*: \H\Z/2_{**}(\1/2_\epsilon) \to \H\Z/2_{**}\H\Z/2$, where $\H\Z/2_{**}(\1/2_\epsilon)$ is free over $\H\Z/2_{**}$ on generators $1 \in \H\Z/2_{0,0}(\1/2_\epsilon), x \in \H\Z/2_{1,0}$ with $\alpha_*(1) = 1$ and $\alpha_*(x) = \tau_0$.
\end{lemma}
\begin{proof}
The computation of $\H\Z/2_{**}(\1/2_\epsilon)$ is clear. In order to establish the effect of $\alpha$ homology, by Lemma~\ref{lemm:mot-hom-cohom-dual} we may dualize and establish the effect on cohomology instead\NB{details}. In other words we need to determine the action of the motivic Steenrod algebra on $\H\Z/2^{**}(\1/2_\epsilon)$. By stability, we may as well do this for $1/2_\epsilon \wedge \Gm$, and then we consider the map $\1/2_\epsilon \wedge \Gm \to \R\P^2$ from Lemma \ref{lemm:pi1-RP2}(1).
We have $H^{**}(\R\P^\infty) \wequi H^{**}\fpsr{u,v}/u^2 = \tau v + \rho u$, and the Steenrod action is $\Sq^1(u) = v$, with all other operations zero, by Lemma \ref{lemm:action-on-Bsigma2}.
By Lemma \ref{lemm:pi1-RP2}(2), our map sends $u$ to $1$ and $v$ to $x$.
Hence $\Sq^1(1) = x$ and all other operations act trivially, which is precisely dual to the map that we claimed (use Theorem \ref{thm:HZ-steenrod-dedekind}(5)).
\end{proof}

\begin{remark}[Zhouli Xu]
This lemma is false in general for $\1/2$ in place of $\1/2_\epsilon$. Indeed still $\H\Z/2_{**}(\1/2)$ is free on generators $1, x$, but $\alpha_*(x) = \tau_0 + \rho \xi_1$.
\end{remark}

Lemma \ref{lemm:1-ss-2-versal} provides us with a  map of normed spectra $\1 \sslash 2 \to \H\Z/2$, which is even unique in this case.
The crux of our splitting result is the following.

\begin{lemma} \label{lemm:1-ss-2-surjection}
Let $2$ be invertible on $S$. The  map of normed spectra 
$\1 \sslash 2 \to \H\Z/2$ induces a surjection on $\H\Z/2_{**}(-)$.
\end{lemma}
\begin{proof}
By Lemma \ref{lemm:normed-2-h-torsion}, the unit map $\1 \to \1 \sslash 2$ extends over the cofiber of $2_\epsilon$, and hence by Lemma \ref{lemm:steenrod-action-moore-spectrum}, the image of $\H\Z/2_{**}(\1 \sslash 2) \to \H\Z/2_{**}\H\Z/2$ contains $\tau_0$. By construction, the image is a subring stable under power operations. The claim now follows from Theorem \ref{thm:bockstein-generates}.
\end{proof}

We thus arrive at the following result.
\begin{theorem} \label{thm:1-ss-2-splits}
Let $1/2 \in S$. Then $\1 \sslash 2$ splits into a wedge of shifts and twists of $\H\Z/2$.
\end{theorem}
\begin{proof}
Since the formation of free normed spectra commutes with arbitrary base change (see Corollary \ref{corr:free-normed-base-change}), we may assume that $S = \Spec(\Z[1/2])$; in particular we may assume that $S$ is connected.

We note that the algebra $\ul{\pi}_{**} \H\Z/2 = \ul{\pi}_{**}(D^0(S^1) \wedge \H\Z/2)$ is connected with respect to any line of slope $1 > q > 0$ through the origin, by Proposition \ref{prop:HZ-dedekind}(3). It thus follows from Proposition \ref{prop:free-normed-spectra} and \cite[Theorem 3.10]{bachmann-MGM} that $\H\Z/2 \wedge \NSym(S^1)$ is a summand of a split Tate motive, and that $\ul{\pi}_{**}(\H\Z/2 \wedge \NSym(S^1))$ is connected with respect to a certain line of slope $1 > q > 1/2$.
Combining this with Lemmas \ref{lemm:1-ss-2-identification} and \ref{lemm:1-ss-2-surjection} we find that Theorem \ref{thm:detect-splitting} applies. The result follows.
\end{proof}

\begin{remark} \NB{Is $\1 \sslash 2 \wequi \1 \sslash 2_\epsilon$?}
The same reasoning applies to show that $\1 \sslash 2_\epsilon$ splits into a wedge of shifts and twists of $\H\Z/2$: by Lemma \ref{lemm:1-ss-2-identification} we know that $\1 \sslash 2_\epsilon \wedge \H\Z/2 \wequi \NSym(S^1) \wedge \H\Z/2$, $\1 \sslash 2_\epsilon \to \H\Z/2$ induces a surjection on $\H\Z/2_{**}$ by the same argument as in Lemma \ref{lemm:1-ss-2-surjection}, and then we conclude again by Milnor-Moore as in Theorem \ref{thm:1-ss-2-splits}.
\end{remark}

% I no longer see the point of this.
\begin{comment}
\begin{corollary}
Under the same assumptions, if $E \in \NAlg(S)$ and $2=0$ in $E$, then the unit map $E \to E \wedge \H\Z/2$ splits (in $\SH(S)$).
\end{corollary}
\begin{proof} \NB{This feels slightly fishy to me.}
Let us note that if $F, G \in \SH(S)$ are such that $G$ is a homotopy commutative ring spectrum and $F$ is a $G$-module in the homotopy category, then the unit map $F \to F \wedge G$ is split by the multiplication (action) $F \wedge G \to F$. By Lemma \ref{lemm:1-ss-2-versal} there is a ring map $\1 \sslash 2_\epsilon \to E$, so $E$ is a module over $\1 \sslash 2_\epsilon$, and in particular the unit map $u_2: E \to E \wedge \1 \sslash 2_\epsilon$ is split. By Corollary \ref{corr:1-ss-2-splits}, $\1 \sslash 2_\epsilon$ is equivalent in $\SH(S)$ to a free module over $\H\Z/2$; in particular the unit map $u_H: \1 \sslash 2_\epsilon \to \1 \sslash 2_\epsilon \wedge \H\Z/2$ splits. Now the composite \[ E \xrightarrow{u_H} E \wedge \H\Z/2 \xrightarrow{u_2} E \wedge \1 \sslash 2_\epsilon \wedge \H\Z/2 \xrightarrow{\id_E \wedge s_H} E \wedge \1 \sslash 2_\epsilon \xrightarrow{s_2} E \] is homotopic to the composite \[ E \xrightarrow{u_2} E \wedge \1 \sslash 2_\epsilon \xrightarrow{u_H} E \wedge \1 \sslash 2_\epsilon \wedge \H\Z/2 \xrightarrow{\id_E \wedge s_H} E \wedge \1 \sslash 2_\epsilon \xrightarrow{s_2} E, \] which is the identity by construction. We have thus constructed a splitting of $u_H$, which concludes the proof.
\end{proof}
\end{comment}

\begin{subappendices}
\section{Some considerations in $\A^1$-algebraic topology}\label{sec:A1topology}
\label{app:topology}
Let $S$ be a base scheme. For $a \in \scr O^\times(S)$, we have the automorphism $\lra{a}: \P^1 \to \P^1, (x:y) \mapsto (ax: y) = (x: a^{-1}y)$. It stabilises to an automorphism of $\1 \in \SH(S)$ which we denote by the same name. Note that $\lra{1} = 1$.

\begin{lemma} \label{lemm:lra-squares}
Suppose that $a \in \scr O^\times(S)$. Then $\lra{a^2} \wequi \id \in [\P^1, \P^1]_{\Spc(S)}$. In particular $\lra{a^2} = 1 \in [\1, \1]_{\SH(S)}$.
\end{lemma}
\begin{proof}
The map $(x:y) \mapsto (a^2x:y)$ is equal to $(x:y) \mapsto (ax:a^{-1}y)$. We have a map $\GL_2(S) \to \Map(\P^1, \P^1)$. It is thus enough to connect the matrix $A = \begin{bmatrix}a^{-1} & 0 \\ 0 & a\end{bmatrix}$ to the identity matrix in $\GL_2(S)$ via $\A^1$-paths. By ``Whitehead's Lemma'', $A$ is a product of elementary matrices: \[ \begin{bmatrix}a^{-1} & 0 \\ 0 & a\end{bmatrix} = \begin{bmatrix}1 & 1/a \\ 0 & 1\end{bmatrix} \begin{bmatrix}1 & 0 \\ 1-a & 1\end{bmatrix}\begin{bmatrix}1 & -1 \\ 0 & 1\end{bmatrix}\begin{bmatrix}1 & 0 \\ 1-a^{-1} & 1\end{bmatrix}. \]
Since elementary matrices can be connected to the identity matrix by evident paths, the result follows.
\end{proof}

Similarly we prove the following.
\begin{lemma} \label{lemm:switch-lra}
The switch map on $T \wedge T$ is $\lra{-1}$.
\end{lemma}
\begin{proof}
The switch map corresponds to the matrix $\begin{pmatrix} 0 & 1 \\ 1 & 0 \end{pmatrix}$. Adding the first column to the second, and then subtracting the second row from the first, then the second column from the first, we obtain $\begin{pmatrix} -1 & 0 \\ 0 & 1 \end{pmatrix}$. This corresponds to $\lra{-1} \wedge \id$. Clearly elementary row operations are homotopic to the identity. The result follows.
\end{proof}

Let 
\[
n_\epsilon = 1 + \lra{-1} + \dots + \lra{\pm 1},
\]
 where the sum consist of $n$ terms.
\begin{lemma} \label{lemm:powering-n-epsilon}
The pointed map $p_n: \Gm \to \Gm, x \mapsto x^n$ induces $\bar{p}_n = \Sigma^\infty(p_n) \wedge \Gm^{-1}: \1 \to \1 \in \SH(S)$ and satisfies $\bar{p}_n \wequi n_\epsilon$.
\end{lemma}
\begin{proof}
See \cite[Proposition 2.1]{bachmann-bott}.
\end{proof}

\begin{lemma} \label{lemm:trace-of-Zi}
Let $1/2 \in S$ and put $S' = S[T]/(T^2 + 1)$. Then $\tr_{S'/S}(1) = \lra{2}2_\epsilon$.\NB{We have $\lra{a}2_\epsilon = 2_\epsilon$, by Lemma 4.3 in BEHKSY.}
\end{lemma}
\begin{proof}
Let $f: \P^1_S \to \P^1_S$ be given by the map informally described as $(x:y) \mapsto (-x^3:y^3)$. Let $Z$ be its fixed locus. We shall apply the Lefschetz fixed point formula \cite[Theorem 1.3]{hoyois2015quadratic} to $f$. This will give us a computation of the categorical trace $\tr(\Sigma^\infty_+ f)$ of $\Sigma^\infty_+ f$. Then we shall compute that trace directly, and deduce the result by comparison.

We first check that the Lefschetz fixed point formula applies. Note that in the affine patch $y=1$, $f$ corresponds to the map $x \mapsto -x^3$, and in the affine patch $x=1$ it corresponds to $y \mapsto -y^3$. The fixed locus in the first patch is $S[T]/(T + T^3)$ which is the disjoint union of $0 \wequi S$ and $S'$, both of which are smooth. Hence the entire fixed locus is $Z \wequi 0 \coprod \infty \coprod S'$, which is in particular smooth. Let $i: Z \hookrightarrow \P^1_S$ be the inclusion. The map $f: \P^1 \to \P^1$ induces $df: \Omega_{\P^1} \to \Omega_{\P^1}$. Let $N_i$ be the conormal bundle; we have an injection $N_i \hookrightarrow i^* \Omega_{\P^1}$ and $df$ restricts to $i^* df: N_i \to N_i$. We need to check that $\id - i^* df$ is an isomorphism. By direct computation we find that $i^*df$ is given by $0$ on $0, \infty$ and by $3$ on $S'$. Since $1/2 \in S$, $\id - i^*df$ is an isomorphism and the Lefschetz fixed point theorem applies. It tells us that \[ \tr(\Sigma^\infty_+ f) = 1 + 1 + \tr_{S'/S}(\lra{-2}). \]

Now we compute $\tr(\Sigma^\infty_+ f)$ directly. We have $\Sigma^\infty \P^1_+ \wequi \1 \vee \P^1$, and $\sigma^\infty_+ f$ corresponds to the diagonal map $\id \vee \Sigma^\infty f$ (this makes sense since $f$ is a pointed map, provided we point $\P^1$ at $\infty$ as usual). It follows from \cite[Lemma 1.7]{may2001additivity} that $\tr(\Sigma^\infty_+ f) = \tr(\id_{\1}) + \tr(\Sigma^\infty f)$. Clearly $\tr(\id_{\1}) = 1$. Denote the canonical isomorphism $[\P^1, \P^1] \to [\1, \1]$ (inverse to smashing with $\id_{\P^1}$) by $D$. By \cite[Proposition 4.14]{dugger2014coherence} we have $\tr(\Sigma^\infty f) = \tr(\id_{\P^1}) D(\Sigma^\infty f)$. It follows from Lemma \ref{lemm:powering-n-epsilon} that $D(\Sigma^\infty f) = \lra{-1} 3_\epsilon$, and from \cite[Proposition 4.21]{dugger2014coherence} (together with Lemma \ref{lemm:switch-lra}) that $\tr(\id_{\P^1}) = \lra{-1}$. Hence we obtain \[ \tr(\Sigma^\infty_+ f) = 1 + 3_\epsilon. \]

Comparing with the previous expression, and using that $\tr_{S'/S}(\lra{-2}) = \lra{-2}\tr_{S'/S}(1)$, we obtain the desired result.
\end{proof}

We also have the following observation.
\begin{lemma} \label{lemm:pi1-RP2}
\begin{enumerate}
\item The inclusion $\Gm \wequi \R\P^1 \hookrightarrow \R\P^2$ factors over the cofiber of $2_\epsilon$.
\item Let $S$ be essentially smooth over a Dedekind domain.
  Then there is a non-zero unique class $u_0 \in H^{1,1}(\Gm/2_\epsilon, \Z/2)$ which restricts to $0$ along $* \to \Gm/2_\epsilon$.
  We have $H^{**}(\Gm/2_\epsilon, \Z/2)\{1, u_0, v_0\}$ where $v_0 = \Sq^1(u_0)$.
  The induced map \[ H^{**}\fpsr{u,v}/u^2 = \tau v + \rho u \wequi H^{**}(\R\P^\infty, \Z/2) \to H^{**}(\Gm/2_\epsilon, \Z/2) \] sends $u$ to $u_0$ and $v$ to $v_0$.
\end{enumerate}
\end{lemma}
\begin{proof}
(1) Consider the commutative diagram
\begin{equation*}
\begin{CD}
\A^1 \setminus 0 @>j>> \A^2 \setminus 0 \\
@VpVV                   @VVV           \\
\R\P^1           @>i>> \R\P^2,
\end{CD}
\end{equation*}
where the horizontal maps are the canonical inclusions and the vertical maps are the canonical quotients. Under the isomorphism $\Gm \wequi \R\P^1$ the map $p$ corresponds to the squaring map, i.e. induces $2_\epsilon$ in (stable) homotopy (by Lemma \ref{lemm:powering-n-epsilon}). We thus need to prove that $ip$ is null homotopic, for which it suffices to show that $j$ is null homotopic. The map $j$ is given by $x \mapsto (x, 0)$. We can first apply the homotopy $(x, t) \mapsto (x, t)$ to homotope it to $x \mapsto (x, 1)$, and then apply $(x, t) \mapsto ((1-t)x, 1)$ to homotope it to $x \mapsto (0, 1)$. This concludes the proof.

(2) Since $2_\epsilon$ vanishes in $\H\Z/2_{**}$ (e.g. by Lemma \ref{lemm:normed-2-h-torsion}) we have \[ \Sigma^\infty_+(\Gm/2_\epsilon) \wedge \H\Z/2 \wequi \H\Z/2 \vee \Sigma^{1,1} \H\Z/2 \vee \Sigma^{2,1}\H\Z/2, \] and the Bockstein $\Sq^1: H^{1,1}(\Gm/2_\epsilon) \to H^{2,1}(\Gm/2_\epsilon)$ is non-zero.
Uniqueness of the generator in degree $(1,1)$ (subject to vanishing in $*$) follows from Proposition \ref{prop:HZ-dedekind}(2) (i.e. the fact that motivic cohomology is understood in weight $1$).
We claim that $\Sq^1(u_0)$ generates the term in degree $(2,1)$.
For this we may assume that $S=\Spec(\Z)$.
Under this assumption, Lemma \ref{lem:ptm}(1) shows that $H^{2,1}(\Gm/2_\epsilon) \wequi \F_2$, so any non-zero element generates; this proves the claim.

The cohomology of $\R\P^\infty$ is determined in e.g. Lemma \ref{lemm:coh-of-BSigma2}.
It remains to determine the map $i^*: H^{**}(\R\P^\infty) \to H^{**}(\Gm/2_\epsilon)$.
By construction one has $\Sq^1(u) = v$ and $u|_* = 0$ \cite[Lemma 10.6]{spitzweck2012commutative}.
It thus suffices to show that $i^*(u) \ne 0$.
This follows from the fact that $u|_{\R\P^1} \ne 0$, as can be verified e.g. using complex realization.
\end{proof}

\section{Base change of normed spectra}\label{sec:basechange}
Recall the notation $Q(\FEt, Z)$ from \S\ref{subapp:kan-extension}.
Also recall the category $\NAlg_\scr{C}(\scr D)$ of $\scr C$-normed objects in $\scr D$ from \cite[Definition 7.1]{norms}.

\begin{proposition} \label{prop:norms-extension-abstract}
Let $S$ be a scheme and $\scr C \subset_\fet \scr C' \subset_\fet \Sch_S$. Assume that for every $Z \in \scr C$ the stack $Q(\FEt, Z) \in \PSh(\scr C')$ is left Kan extended from its restriction to $\scr C$. Let $\scr D$ be a presentably normed $\infty$-category over $S$. Then the canonical forgetful functor $\NAlg_{\scr C'}(\scr D) \to \NAlg_{\scr C}(\scr D)$ is an equivalence.
\end{proposition}
\begin{proof}
The proof is essentially the same as that of \cite[Proposition 7.6(6)]{norms}.
For the convenience of the reader we briefly recall the argument.
The restriction $\NAlg_{\scr C'}(\scr D) \to \NAlg_{\scr C}(\scr D)$ is conservative, and so it suffices to prove that its left adjoint is fully faithful.
Recalling that normed spectra embed into a category of sections, one notes that for formal reasons the restriction functor on the section categories admits a fully faithful left adjoint $L$.
It thus suffices to prove that $L$ preserves the subcategories of normed spectra.
Using an explicit formula for $L$ (as a relative left Kan extension), one reduces to proving that a certain inclusion of indexing categories is cofinal.
Using Quillen's theorem B, we finally arrive at the following: we seek to prove that \[ \scr E := (\scr C^\op)_{/X} \times_{\Span(\scr C,\all,\fet)_{/X}} (\Span(\scr C,\all,\fet)_{/X})_{(Z\stackrel f\leftarrow Y\stackrel p\to X)/} \] is a weakly contractible category, for all $X \in \scr C', Z \in \scr C$ and $Y \in \scr C'$ finite étale over $X$. Consider the functor $F: \scr C_0 := (\scr C^\op)_{/X} \to \Spc, X_0 \mapsto Q(\FEt, Z)(X_0)$. There is a natural transformation $F \Rightarrow Q(\FEt, Z)(X)$ (to a constant functor). Form the cartesian square
\begin{equation*}
\begin{CD}
F' @>>> F \\
@VVV  @VVV \\
\{Z\stackrel f\leftarrow Y\stackrel p\to X\} @>>> Q(\FEt, Z)(X).
\end{CD}
\end{equation*}
By inspection, $\scr E$ is the category of elements (or Grothendieck construction, lax-slice category, unstraightening) $Gr(F')$. $\scr E$ being weakly contractible means that its nerve or \emph{groupoid completion} $\scr E^{gpd} \wequi * \in \Spc$. For any functor $G: \scr C_0 \to \Spc$ we have $Gr(G)^{gpd} \wequi \colim_{\scr C_0} G$ (this follows from \cite[Corollary 3.3.4.3]{HTT}). Our task is thus to show that $\colim_{\scr C_0} F' \wequi *$. By universality of colimits (in $\Spc$) for this it is enough to show that the canonical map $\colim_{\scr C_0} F \to Q(\FEt, Z)(X)$ is an equivalence. This is exactly the assumption that $Q(\FEt, Z) \in \PSh(\scr C')$ is obtained by left Kan extension from its restriction to $\scr C$.
\end{proof}

\def\Et{\cat{E}\mathrm{t}}
\begin{lemma} \label{lemm:NAlg-sheaf}
Suppose that $S$ is a scheme, $\scr D$ a presentably normed $\infty$-category over $S$ satisfying Zariski (respectively Nisnevich) descent, and $\scr C \subset_\fet \Sch_S$.
Then $\NAlg_{\scr C}(\scr D(\ph))$\NB{notation a bit sloppy} satisfies Zariski (respectively Nisnevich) descent.
\end{lemma}
\begin{proof}
Consider the presheaf of categories on $\Et_S$ \[ F: U \mapsto \Sect(\Span(\scr C_{/U}, \all, \fet), \scr D). \]
Let $U \to S$ be a covering.
We obtain an adjunction\todo{details?} \[ a^*: F(S) \adj \lim_\Delta F(U_\bullet): a_*. \]
The functors are informally described as \[ (a^* E)_n(X \to U_n) \wequi E(X \to U_n \to S) \] and \[ (a_* E_\bullet)(X) \wequi \lim_{n \in \Delta} E_n(X \times_S U_n). \]
By construction $\NAlg_{\scr C}(\scr D(U))$ embeds into $F(U)$.
It is enough to prove that if $E \in \NAlg_{\scr C}(\scr D(S)) \subset F(S)$ and \[ F_\bullet \in \lim_\Delta \NAlg_{\scr C}(\scr D(U_\bullet)) \subset \lim_\Delta F(U_\bullet) \] then $a_*a^* E \wequi E$ and $a^*a_* F_\bullet \wequi F_\bullet$.
Since $E$ is cartesian over backwards morphisms, we get $E(X \times U_n) \wequi E(X)|_{X \times U_n}$, and so on.
Using this, the desired equivalences follow (via the explicit description of $a^*, a_*$) from the assumption that $\scr D$ satisfies descent.
\end{proof}

\begin{corollary} \label{corr:norms-extension}
Let $S$ be a scheme and $\scr D$ a presentably normed $\infty$-category over $S$, satisfying Zariski-descent. Then $\NAlg_{\Sm_S}(\scr D) \wequi \NAlg_{\Sch_S}(\scr D)$.
\end{corollary}
\begin{proof}
Since $X \mapsto \scr D(X)$ is a sheaf in the Zariski topology, so are $X \mapsto \NAlg_{\Sm_X}(\scr D)$ and $X \mapsto \NAlg_{\Sm_X}(\scr D)$, by Lemma \ref{lemm:NAlg-sheaf}. Consequently we may assume that $S$ is affine. By \cite[Proposition 7.6(5)]{norms} we have $\NAlg_{\Sm_S}(\scr D) \wequi \NAlg_{\SmAff_S}(\scr D)$ and $\NAlg_{\Sch_S}(\scr D) \wequi \NAlg_{\Aff_S}(\scr D)$. Thus we may replace $\Sm_S$ by $\SmAff_S$ and $\Sch_S$ by $\Aff_S$ in the statement. The result now follows from Proposition \ref{prop:norms-extension-abstract} and Corollary \ref{corr:FEt-base-change}.
\end{proof}

\begin{corollary}
Let $f: S' \to S$ be a morphism of schemes and $\scr D$ a presentably normed $\infty$-category over $S$ satisfying Zariski descent. Then the functor $f^*: \scr D(S) \to \scr D(S')$ lifts to a colimit-preserving functor $f^*: \NAlg_{\Sm_S}(\scr D) \to \NAlg_{\Sm_{S'}}(\scr D)$ in the sense that the following diagram commutes
\[
\begin{tikzcd}
\NAlg_{\Sm_S}(\scr D) \ar{r}{f^*} \ar{d} &  \NAlg_{\Sm_{S'}}(\scr D) \ar{d}\\
\scr D(S) \ar{r}{f^*} & \scr D(S'),
\end{tikzcd}
\]
where the vertical arrows are forgetful functors.

Furthermore, the top arrow has right adjoint $f_N$, and if $\scr D$ satisfies pro-smooth and proper base change, then $f^*: \NAlg_{\Sm_S}(\scr D) \adj \NAlg_{\Sm_{S'}}(\scr D): f_N$ lifts $f^*: \scr D(S) \adj \scr D(S'): f_*$in the sense that
\begin{equation} \label{eq:fN-U-comm}
\begin{tikzcd}
\NAlg_{\Sm_S}(\scr D)  \ar{d} &  \NAlg_{\Sm_{S'}}(\scr D) \ar{d} \ar[swap]{l}{f_*}\\
\scr D(S) & \scr D(S')\ar[swap]{l}{f_*} ,
\end{tikzcd}
\end{equation}
commutes, where the vertical arrows are forgetful functors.
\end{corollary}
\begin{proof}
By Corollary \ref{corr:norms-extension}, we may replace $\Sm_S$ by $\Sch_S$ and $\Sm_{S'}$ by $\Sch_{S'}$. Then existence of the functor is \cite[Proposition 7.6(7)]{norms}. The forgetful functor $\NAlg_{\Sm_S}(\scr D) \to \scr D(S)$ is conservative and preserves sifted colimits and smash products, and similarly for $S'$. Since finite coproducts in $\NAlg_{\Sm_S}(\scr D)$ are given by finite smash products, and the functor $f^*: \scr D(S) \to \scr D(S')$ preserves colimits and smash products, we conclude that $f^*: \NAlg_{\Sm_S}(\scr D) \to \NAlg_{\Sm_{S'}}(\scr D)$ preserves sifted colimits and finite coproducts. It follows that it preserves all colimits \cite[Lemma 2.8]{norms}. Thus there is a right adjoint $f_N$, by \cite[Proposition 7.6(1)]{norms} and the adjoint functor theorem.

It remains to identify $f_N$, that is, we must show that the exchange transformation $UF_N \to f_*U$ (where $U$ denotes the forgetful functors) obtained from $Uf^* \wequi f^*U$ is an equivalence.
The formation of these exchange transformations is compatible with composition\NB{ref/details?}, and it is an equivalence for $f$ smooth \cite[Theorem 8.2]{norms}.
Using that $\scr D$ satisfies Zariski descent, this implies that the problem is local on $S$ and $S'$, so we may assume that both are affine.
The functor $f_N$ lifts $f_*$ if $f$ is proper or pro-smooth \cite[Theorem 8.2, Example 8.4]{norms}. Any morphism of affine schemes is a composite of a closed immersion (hence a proper morphism) and a pro-smooth morphism (projection from an infinite dimensional affine space); hence $f_N$ lifts $f_*$ in general.
\end{proof}

\begin{remark}
These assumptions are satisfied for example when $\scr D = \SH$.
\end{remark}

\begin{corollary} \label{corr:free-normed-base-change}
Under the same assumptions, the formation of free normed objects commutes with arbitrary base change.
\end{corollary}
\begin{proof}
Take left adjoints in the commutative square \eqref{eq:fN-U-comm}.
\end{proof}
\end{subappendices}

%% file: splitting-sa.bbl
\providecommand{\bysame}{\leavevmode\hbox to3em{\hrulefill}\thinspace}
\providecommand{\MR}{\relax\ifhmode\unskip\space\fi MR }
\providecommand{\MRhref}[2]{%
  \href{http://www.ams.org/mathscinet-getitem?mr=#1}{#2}
}
\providecommand{\href}[2]{#2}
\begin{thebibliography}{BMMS86}

\bibitem[Bac16]{bachmann-hurewicz}
Tom Bachmann, \emph{On the conservativity of the functor assigning to a motivic
  spectrum its motive},
  \href{https://arxiv.org/abs/1506.07375}{arXiv:1506.07375}.

\bibitem[Bac22]{bachmann-MGM}
\bysame, \emph{Motivic spectral mackey functors},
  \href{https://arxiv.org/abs/2205.13926}{arXiv:2205.13926}, 2022.

\bibitem[BEH21]{colimits}
Tom Bachmann, Elden Elmanto, and Jeremiah Heller, \emph{Motivic colimits and
  extended powers}, 2021.

\bibitem[BEH22]{operations}
\bysame, \emph{Normed motivic spectra and power operations}, 2022.

\bibitem[BE{\O}20]{bachmann-bott}
Tom Bachmann, Elden Elmanto, and Paul~Arne {{\O}stv{\ae}r}, \emph{Stable
  motivic invariants are eventually étale local},
  \href{https://arxiv.org/abs/2003.04006}{arXiv:2003.04006}, 2020.

\bibitem[BH20]{norms}
Tom Bachmann and Marc Hoyois, \emph{Norms in motivic homotopy theory},
  \href{https://arxiv.org/abs/1711.03061}{arXiv:1711.03061}, to appear in
  Ast\'erisque.

\bibitem[BMMS86]{hinfty}
R.~R. Bruner, J.~P. May, J.~E. McClure, and M.~Steinberger, \emph{{$H\sb \infty
  $} ring spectra and their applications}, viii+388. \MR{836132}

\bibitem[Dug14]{dugger2014coherence}
Daniel Dugger, \emph{Coherence for invertible objects and multigraded homotopy
  rings}, Algebraic \& Geometric Topology \textbf{14} (2014), no.~2,
  1055--1106.

\bibitem[Eis13]{eisenbud2013commutative}
David Eisenbud, \emph{Commutative algebra: with a view toward algebraic
  geometry}, vol. 150, Springer Science \& Business Media, 2013.

\bibitem[EK18]{elmanto2018perfection}
Elden Elmanto and Adeel~A Khan, \emph{Perfection in motivic homotopy theory},
  arXiv preprint arXiv:1812.07506 (2018).

\bibitem[Hoy15]{hoyois2015quadratic}
Marc Hoyois, \emph{A quadratic refinement of the
  grothendieck--lefschetz--verdier trace formula}, Algebraic \& Geometric
  Topology \textbf{14} (2015), no.~6, 3603--3658.

\bibitem[Lur17]{HTT}
Jacob Lurie, \emph{Higher topos theory}, April 2017.

\bibitem[May01]{may2001additivity}
J~Peter May, \emph{The additivity of traces in triangulated categories},
  Advances in Mathematics \textbf{163} (2001), no.~1, 34--73.

\bibitem[Rav86]{ravenel1986complex}
Douglas~C Ravenel, \emph{Complex cobordism and stable homotopy groups of
  spheres}, vol. 121, Academic press New York, 1986.

\bibitem[Spi12]{spitzweck2012commutative}
Markus Spitzweck, \emph{A commutative $\mathbb{P}^1$-spectrum representing
  motivic cohomology over dedekind domains}, arXiv preprint arXiv:1207.4078
  (2012).

\end{thebibliography}
